\documentclass[reqno,12pt]{amsart}
\usepackage{amsmath,amsthm,amssymb,amsfonts,amscd}
\usepackage{epsfig}
\usepackage{xypic}
\numberwithin{equation}{section}

\theoremstyle{plain}
\newtheorem{theorem}{Theorem}

\newtheorem{lemma}[theorem]{Lemma}
\newtheorem{corollary}[theorem]{Corollary}
\newtheorem{proposition}[theorem]{Proposition}
\newtheorem*{theorem*}{Theorem}
\newtheorem*{conjecture*}{Conjecture}

\theoremstyle{definition}
\newtheorem{remark}[theorem]{Remark}
\newtheorem{example}[theorem]{Example}
\newtheorem*{definition}{Definition}

\newcommand{\CC}{{\mathbb{C}}}

\newcommand{\PP}{{\mathbb{P}}}

\newcommand{\ZZ}{{\mathbb{Z}}}

\begin{document}
\title{Strange duality of weighted homogeneous polynomials}
\author{Wolfgang Ebeling and Atsushi Takahashi}
\address{Institut f\"ur Algebraische Geometrie, Leibniz Universit\"at Hannover, Postfach 6009, D-30060 Hannover, Germany}
\email{ebeling@math.uni-hannover.de}
\address{
Department of Mathematics, Graduate School of Science, Osaka University, 
Toyonaka Osaka, 560-0043, Japan}
\email{takahashi@math.sci.osaka-u.ac.jp}
\subjclass[2010]{14J33, 32S25, 14L30}
\begin{abstract}
We consider a mirror symmetry between invertible weighted homogeneous polynomials in three variables. We define Dolgachev and Gabrielov numbers for them and show that we get a duality between these polynomials generalizing Arnold's strange duality between the 14 exceptional unimodal singularities.
\end{abstract}
\maketitle
\section*{Introduction}
Mirror symmetry is now understood as a categorical duality between 
algebraic geometry and symplectic geometry.
One of our motivations is to apply some ideas of mirror symmetry to singularity theory 
in order to understand various mysterious correspondences among
isolated singularities, root systems, Weyl groups, Lie algebras, 
discrete groups, finite dimensional algebras and so on.
In this paper, we shall generalize Arnold's strange duality 
for the 14 exceptional unimodal singularities to a specific class 
of weighted homogeneous polynomials in three variables called invertible polynomials.
Let $f(x,y,z)$ be a polynomial which has an isolated singularity 
only at the origin $0\in\CC^3$. 
A distinguished basis of vanishing (graded) Lagrangian submanifolds 
in the Milnor fiber of $f$ can be categorified to an $A_\infty$-category ${\rm Fuk}^{\to}(f)$
called the directed Fukaya category whose derived category $D^b{\rm Fuk}^\to(f)$
is, as a triangulated category, an invariant of the polynomial $f$. 
Note that the triangulated category $D^b{\rm Fuk}^\to(f)$ has a full exceptional collection.
If $f(x,y,z)$ is a weighted homogeneous polynomial 
then one can consider another interesting triangulated category, the category 
of a maximally-graded singularity $D^{L_f}_{Sg}(R_f)$:
\begin{equation}
D^{L_f}_{Sg}(R_f):=D^b({\rm gr}^{L_f}\text{-}R_f)/D^b({\rm proj}^{L_f}\text{-}R_f),
\end{equation}
where $R_f:=\CC[x,y,z]/(f)$ and $L_f$ is the maximal grading of $f$ (see Section 1).
This category $D^{L_f}_{Sg}(R_f)$ is considered as an analogue 
of the bounded derived category of coherent sheaves on a smooth proper algebraic variety.
It is known that the Berglund--H\"{u}bsch transpose for some polynomials with nice properties 
induces the topological mirror symmetry which gives the systematic construction 
of mirror pairs of Calabi--Yau manifolds. 
Therefore, we may expect 
that the topological mirror symmetry can also be categorified to the following:
\begin{conjecture*}[\cite{t:1}\cite{t:2}]
Let $f(x,y,z)$ be an invertible polynomial $($see Section 1 for the definition$)$.
\begin{enumerate}
\item There should exist a quiver with relations $(Q,I)$ and triangulated equivalences
\begin{equation}\label{hms:1}
D^{L_f}_{Sg}(R_f)\simeq D^b({\rm mod}\text{-}\CC Q/I)\simeq D^b{\rm Fuk}^\to(f^t).
\end{equation}
\item There should exist a quiver with relations $(Q',I')$ and triangulated equivalences
\begin{equation}\label{hms:2}
D^b{\rm coh}({\mathcal C}_{G_f})\simeq D^b({\rm mod}\text{-}\CC Q'/I')\simeq
D^b{\rm Fuk}^\to (T_{\gamma_1,\gamma_2,\gamma_3}),
\end{equation}
which should be compatible with the triangulated equivalence \eqref{hms:1},
where ${\mathcal C}_{G_f}$ is the weighted projective line associated to the 
maximal abelian symmetry group $G_f$ of $f$ and $T_{\gamma_1,\gamma_2,\gamma_3}$ 
is a ``cusp singularity" $($see Section 3$)$.
\end{enumerate}
\end{conjecture*}
There are many evidences of the above conjectures which follow
from related results by several authors.
Among them, the most important one in this paper 
is that one should be able to choose as $(Q',I')$ 
the quiver obtained by the graph $T(\gamma_1,\gamma_2,\gamma_3)$ in Section 3 
with a suitable orientation together with two relations along the dotted edges.
This leads us to our main theorem, the strange duality for invertible 
polynomials:
\begin{theorem*}[Theorem \ref{thm:main}]
Let $f(x,y,z)$ be an invertible polynomial.
The Dolgachev numbers $(\alpha_1,\alpha_2,\alpha_3)$ for $G_f$, 
the orders of isotropy of the weighted projective line ${\mathcal C}_{G_f}$, 
coincide with the Gabrielov numbers $(\gamma_1,\gamma_2,\gamma_3)$ for $f^t$, 
the index of the ``cusp singularity" $T_{\gamma_1,\gamma_2,\gamma_3}$ associated to 
$f^t(x,y,z)+xyz$.
\end{theorem*}
\smallskip
We give here an outline of the paper.
Section 1 introduces the definition of invertible polynomials and
their maximal abelian symmetry groups.
We also recall the Berglund--H\"{u}bsch transpose of invertible polynomials, 
which plays an essential role in this paper.
In Section 2, we first give the classification of invertible polynomials in 3 variables.
Most of the results in this paper rely on this classification of invertible polynomials  
and several data given explicitly by them. 
The main purpose of this section is to define the Dolgachev numbers. 
We associate to each pair of an invertible polynomial $f$ and 
its maximal abelian symmetry group $G_f$ a quotient stack ${\mathcal C}_{G_f}$.
We show in both a categorical way and a geometrical way 
that this quotient stack is a weighted projective line with 
three isotropic points of orders $\alpha_1,\alpha_2,\alpha_3$. 
The numbers $A_{G_f}:=(\alpha_1,\alpha_2,\alpha_3)$ are our Dolgachev numbers.
In Section 3, we associate to an invertible polynomial $f(x,y,z)$ the Gabrielov numbers 
$\Gamma_{f}:=(\gamma_1,\gamma_2,\gamma_3)$ by the ``cusp singularity" 
$T_{\gamma_1,\gamma_2,\gamma_3}$ obtained as the deformation of the polynomial $f(x,y,z)+xyz$. 
Note that the triple of Gabrielov numbers is not an invariant of the singularity 
defined by $f$ but an invariant of the polynomial $f$. 
Section 4 gives the main theorem of this paper, 
a generalization of Arnold's strange duality between the 14 exceptional unimodal 
singularities.
We show that $A_{G_f}=\Gamma_{f^t}$ and $A_{G_{f^t}}=\Gamma_f$ for 
all invertible polynomials $f(x,y,z)$. 
This means that strange duality is one aspect of a mirror symmetry among 
the isolated hypersurface singularities with good group actions.
Therefore, it is now a ``charm" duality and no more a ``strange" duality. 
In Section 5, we collect some additional features of the duality. We show the coincidence of certain invariants for dual invertible polynomials. An additional feature of Arnold's strange duality is a duality between the characteristic polynomials of the Milnor monodromy discovered by K.~Saito. Moreover, the first author observed a relation of these polynomials with the Poincar\'e series of the coordinate rings. We discuss to which extent these facts continue to hold for our duality.

In Section 6, we show how our results fit into the classification of singularities. We recover Arnold's strange duality between the 14 exceptional unimodal singularities. We obtain a new strange duality embracing the 14 exceptional bimodal singularities and some other ones. Note that this duality depends on the chosen invertible polynomials for the 14 exceptional bimodal singularities. In \cite{KPABR} a slightly different version of a duality for these singularities is considered. Finally we discuss how our Gabrielov numbers are related to Coxeter-Dynkin diagrams of the singularities.

\smallskip
{\bf Acknowledgments}.\  
This work is supported 
by the DFG-programme ''Representation Theory'' (Eb 102/6-1).
The second named author is also supported 
by Grant-in Aid for Scientific Research 
grant numbers 20360043 from the Ministry of Education, 
Culture, Sports, Science and Technology, Japan. 
\section{Invertible polynomials}

\begin{sloppypar}

Let $f(x_1,\dots, x_n)$ be a weighted homogeneous complex polynomial. 
This means that there are positive integers $w_1,\dots ,w_n$ and $d$ such that 
$f(\lambda^{w_1} x_1, \dots, \lambda^{w_n} x_n) = \lambda^d f(x_1,\dots ,x_n)$ 
for $\lambda \in \CC^\ast$. We call $(w_1,\dots ,w_n;d)$ a system of {\em weights}. 
If ${\rm gcd}(w_1,\dots ,w_n;d)=1$, then a system of weights is called {\em reduced}. 
A system of weights which is not reduced is called {\em non-reduced}. 
We shall also consider non-reduced systems of weights in this paper. 

\end{sloppypar}

\begin{definition}
A weighted homogeneous polynomial $f(x_1,\dots ,x_n)$ is called {\em invertible} if 
the following conditions are satisfied:
\begin{enumerate}
\item a system of weights $(w_1,\dots ,w_n;d)$ can be uniquely determined by 
the polynomial $f(x_1,\dots ,x_n)$ up to a constant factor ${\rm gcd}(w_1,\dots ,w_n;d)$,
\item $f(x_1,\dots ,x_n)$ has a singularity only at the origin $0\in\CC^n$ which is isolated.
Equivalently, the {\em Jacobian ring} $Jac(f)$ of $f$ defined by
\[  
Jac(f):=\CC[x_1,\dots ,x_n]\left/\left(\frac{\partial f}{\partial x_1},\dots ,
\frac{\partial f}{\partial x_n}\right)\right.
\]
is a finite dimensional algebra over $\CC$ and $\dim_\CC Jac(f)\ge 1$,
\item the number of variables ($=n$) coincides with the number of monomials 
in the polynomial $f(x_1,\dots x_n)$, 
namely, 
\[
f(x_1,\dots ,x_n)=\sum_{i=1}^na_i\prod_{j=1}^nx_j^{E_{ij}}
\]
for some coefficients $a_i\in\CC^\ast$ and non-negative integers 
$E_{ij}$ for $i,j=1,\dots, n$.
\end{enumerate}
\end{definition}
\begin{definition}
Let $f(x_1,\dots ,x_n)=\sum_{i=1}^na_i\prod_{j=1}^nx_j^{E_{ij}}$ be an invertible polynomial.
The {\em Berglund--H\"{u}bsch transpose} $f^t(x_1,\dots ,x_n)$ 
of $f(x_1,\dots ,x_n)$ is the invertible polynomial given by
\[
f^t(x_1,\dots ,x_n)=\sum_{i=1}^na_i\prod_{j=1}^nx_j^{E_{ji}}.
\]
\end{definition}
\begin{definition}
Let $f(x_1,\dots ,x_n)=\sum_{i=1}^na_i\prod_{j=1}^nx_j^{E_{ij}}$ be an invertible polynomial.
Consider the free abelian group $\oplus_{i=1}^n\ZZ\vec{x_i}\oplus \ZZ\vec{f}$ 
generated by the symbols $\vec{x_i}$ for the variables $x_i$ for $i=1,\dots, n$
and the symbol $\vec{f}$ for the polynomial $f$.
The {\em maximal grading} $L_f$ of the invertible polynomial $f$ 
is the abelian group defined by the quotient 
\[
L_f:=\bigoplus_{i=1}^n\ZZ\vec{x_i}\oplus \ZZ\vec{f}\left/I_f\right.,
\]
where $I_f$ is the subgroup generated by the elements 
\[
\vec{f}-\sum_{j=1}^nE_{ij}\vec{x_j},\quad i=1,\dots ,n.
\]
\end{definition}
Note that $L_f$ is an abelian group of rank $1$ which is not necessarily free.
\begin{definition}
Let $f(x_1,\dots ,x_n)$ be an invertible polynomial and $L_f$ be the maximal grading of $f$.
The {\em maximal abelian symmetry group} $G_f$ of $f$ is the abelian group defined by 
\[
G_f:=\left\{(\lambda_1,\dots ,\lambda_n)\in(\CC^\ast)^n \, \left| \,
\prod_{j=1}^n \lambda_j ^{E_{1j}}=\dots =\prod_{j=1}^n\lambda_j^{E_{nj}}\right\} \right.
\]
\end{definition}
Note that the polynomial $f$ is homogeneous with respect to the natural action of 
$G_f$ on the variables. Namely, we have 
\[
f(\lambda_1 x_1, \dots, \lambda_n x_n) = \lambda f(x_1,\dots ,x_n)
\]
for $(\lambda_1,\dots ,\lambda_n)\in G_f$ where 
$\lambda :=\prod_{j=1}^n \lambda_j ^{E_{1j}}=\dots =\prod_{j=1}^n\lambda_j^{E_{nj}}$.
\section{Dolgachev numbers for pairs $(f,G_f)$}

In this paper, we shall only consider invertible polynomials in three variables.
We have the following classification result (see \cite[13.2.]{AGV85}):
\begin{proposition}
Let $f(x,y,z)$ be an invertible polynomial in three variables. 
Then, by a suitable weighted homogeneous coordinate change, 
$f$ becomes one of the five types in Table \ref{TabAppendix1}.
\begin{table}[h]
\begin{center}
\begin{tabular}{|c|c||c||c|}
\hline
{\rm Type} & {\rm Class} & $f$ & $f^t$ \\
\hline
\hline
I & I & $x^{p_1}+y^{p_2}+z^{p_3}$ & $x^{p_1}+y^{p_2}+z^{p_3}$\\
& & $(p_1,p_2,p_3\in\ZZ_{\ge 2})$ & $(p_1,p_2,p_3\in\ZZ_{\ge 2})$ \\
\hline
II & II & $x^{p_1}+y^{p_2}+yz^{\frac{p_3}{p_2}}$ & $x^{p_1}+y^{p_2}z+z^{\frac{p_3}{p_2}}$\\
& & $(p_1,p_2,\frac{p_3}{p_2}\in\ZZ_{\ge 2})$ & $(p_1,p_2,\frac{p_3}{p_2}\in\ZZ_{\ge 2})$\\
\hline
III & IV& $x^{p_1}+zy^{q_2+1}+yz^{q_3+1}$ & $x^{p_1}+zy^{q_2+1}+yz^{q_3+1}$\\
& & $(p_1\in\ZZ_{\ge 2}$, $q_2,q_3\in\ZZ_{\ge 1})$ 
& $(p_1\in\ZZ_{\ge 2}$, $q_2,q_3\in\ZZ_{\ge 1})$\\
\hline
IV & V& $x^{p_1}+xy^{\frac{p_2}{p_1}}+yz^{\frac{p_3}{p_2}}$ 
& $x^{p_1}y+y^{\frac{p_2}{p_1}}z+z^{\frac{p_3}{p_2}}$\\
& & $(p_1,\frac{p_2}{p_1},\frac{p_3}{p_2}\in\ZZ_{\ge 2})$
& $(p_1,\frac{p_2}{p_1},\frac{p_3}{p_2}\in\ZZ_{\ge 2})$\\
\hline
V & VII & $x^{q_1}y+y^{q_2}z+z^{q_3}x$ & $zx^{q_1}+xy^{q_2}+yz^{q_3}$\\
& & $(q_1,q_2,q_3\in\ZZ_{\ge 1})$ & $(q_1,q_2,q_3\in\ZZ_{\ge 1})$  \\
\hline
\end{tabular}
\end{center}
\caption{Invertible polynomials in $3$-variables}\label{TabAppendix1}
\end{table}
\qed
\end{proposition}
In Table \ref{TabAppendix1}, we follow the notation in \cite{s:2}.
Note that the classes in \cite{AGV85} differ from our types,  
the equivalence is given in Table \ref{TabAppendix1}.

We can naturally associate to an invertible polynomial $f(x,y,z)$ 
the following quotient stack:
\begin{equation}
{\mathcal C}_{G_f}:=\left[\CC^3\backslash\{0\}\left/G_f\right.\right]
\end{equation}
Since $f$ has an isolated singularity only at the origin $0\in\CC^3$ 
and $G_f$ is an extension of a one dimensional torus $\CC^\ast$ 
by a finite abelian group, the stack ${\mathcal C}_{G_f}$ is a Deligne--Mumford stack 
and may be regarded as a smooth projective curve 
with a finite number of isotropic points on it. 
Moreover, we have the following:
\begin{theorem}\label{thm:orbifold} 
Let $f(x,y,z)$ be an invertible polynomial.
The quotient stack ${\mathcal C}_{G_f}$ is a smooth projective line $\PP^1$ 
with at most three isotropic points of orders $\alpha_1,\alpha_2,\alpha_3$ 
given in Table \ref{TabAppendix2}.
\begin{table}[h]
\begin{center}
\begin{tabular}{|c||c|c|}
\hline
{\rm Type} & $f(x,y,z)$ & $\left(\alpha_1,\alpha_2,\alpha_3\right)$ \\
\hline
\hline
I & $x^{p_1}+y^{p_2}+z^{p_3}$ & $\left(p_1,p_2,p_3\right)$\\
\hline
II & $x^{p_1}+y^{p_2}+yz^{\frac{p_3}{p_2}}$ & $\left(p_1,\frac{p_3}{p_2},(p_2-1)p_1\right)$\\
\hline
III & $x^{p_1}+zy^{q_2+1}+yz^{q_3+1}$ & $\left(p_1,p_1q_2,p_1q_3\right)$\\
\hline
IV & $x^{p_1}+xy^{\frac{p_2}{p_1}}+yz^{\frac{p_3}{p_2}}$ 
& $\left(\frac{p_3}{p_2}, (p_1-1)\frac{p_3}{p_2}, p_2-p_1+1\right)$\\
\hline
V & $x^{q_1}y+y^{q_2}z+z^{q_3}x$ & $\left(q_2q_3-q_3+1,q_3q_1-q_1+1,q_1q_2-q_2+1\right)$\\
\hline
\end{tabular}
\end{center}
\caption{Dolgachev numbers for pairs $(f,G_f)$}\label{TabAppendix2}
\end{table}
\end{theorem}
\begin{remark}
If the number of isotropic points is less than three
then we set $\alpha_i=1$ for some $i=1,2,3$ in a suitable way for convenience.
\end{remark}
\begin{definition}
The numbers $(\alpha_1,\alpha_2,\alpha_3)$ in Theorem \ref{thm:orbifold} 
are called the {\em Dolgachev numbers} of the pair $(f,G_f)$ and the tuple is denoted by 
$A_{G_f}$.
\end{definition}
\begin{proof}[Proof of Theorem \ref{thm:orbifold}]
There are two ways of the proof of the statement. 
One is categorical and the other is geometrical. 
First, we give a proof by the use of abelian categories of coherent sheaves  
which is already announced in some places (see \cite[Proposition 30]{t:2} for example). 
It is almost the same proof given in \cite[Theorem 5.1.]{t:1} 
where the case $L_f\simeq \ZZ$ is considered.
Following Geigle--Lenzing \cite{gl:1}, 
to a tuple of numbers $A_{G_f}=(\alpha_1,\alpha_2, \alpha_3)$ 
one can associate the ring 
\[
R_{A_{G_f}}:=\CC[X_1,X_2,X_3]\left/(X_1^{\alpha_1}+X_2^{\alpha_2}+X_3^{\alpha_3})\right..
\]
Since $R_{A_{G_f}}$ is graded with respect to an abelian group
\[
L_{A_{G_f}}:=\bigoplus_{i=1}^3\ZZ\vec{X}_i\left/
\left(\alpha_i\vec{X}_i-\alpha_j\vec{X}_j;1\le i<j\le 3\right)\right. ,
\]
one can consider the quotient stack
\[
{\mathcal C}_{A_{G_f}}:=\left[{\rm Spec}(R_{A_{G_f}})\backslash\{0\}/
{\rm Spec}(\CC L_{A_{G_f}})\right].
\]
The quotient stack ${\mathcal C}_{A_{G_f}}$ is a Deligne--Mumford stack 
which may also be regarded as a smooth projective line $\PP^1$ 
with at most three isotropic points of orders $\alpha_1,\alpha_2,\alpha_3$.
It is easy to see this since $R_{A_{G_f}}$ contains 
the ring  $\CC[X_1^{\alpha_1},X_2^{\alpha_2}]$ as a sub-ring,
Now, the statement of Theorem \ref{thm:orbifold} follows from the following
(see also \cite[Theorem 5.1.]{t:1}):
\begin{proposition}
The $L_f$-graded ring $R_{f}:=\CC[x,y,z]/(f)$ can be naturally embedded into 
the $L_{A_{G_f}}$-graded ring $R_{A_{G_f}}$. 
This embedding induces an equivalence of abelian categories$:$
\begin{equation}
{\rm mod}^{L_f}\text{-}R_{f}\left/{\rm tor}^{L_f}\text{-}R_{f}\right.
\simeq 
{\rm mod}^{L_{A_{G_f}}}\text{-}R_{{A_{G_f}}}\left/{\rm tor}^{L_{A_{G_f}}}
\text{-}R_{A_{G_f}}\right..
\end{equation}
In other words, there is an equivalence of abelian categories$:$
\begin{equation}
{\rm coh}({\mathcal C}_{G_f}) \simeq {\rm coh}({\mathcal C}_{A_{G_f}}).
\end{equation}
\end{proposition}
\begin{proof}
The proof is the same as the one in \cite{gl:1} and \cite{t:1} 
except the definition of the map $R_f\hookrightarrow R_{A_{G_f}}$ 
given by Table \ref{TabGenerator}.
\begin{table}[h]
\begin{center}
\begin{tabular}{|c||c|c|c|c|}
\hline
Type & $f(x,y,z)$ & $x$ & $y$ & $z$ \\
\hline
\hline
I & $x^{p_1}+y^{p_2}+z^{p_3}$ & $X_1$ & $X_2$ & $X_3$ \\
\hline
II & $x^{p_1}+y^{p_2}+yz^{\frac{p_3}{p_2}}$ & $X_1X_3$ & $X_3^{p_1}$ & $X_2$ \\
\hline
III & $x^{p_1}+zy^{q_2+1}+yz^{q_3+1}$ & $X_1X_2X_3$ & $X_2^{p_1}$ & $X_3^{p_1}$\\
\hline
IV & $x^{p_1}+xy^{\frac{p_2}{p_1}}+yz^{\frac{p_3}{p_2}}$ 
& $X_2^{\frac{p_3}{p_2}}X_3$ & $X_3^{p_1}$ & $X_1X_2$ \\
\hline
V & $x^{q_1}y+y^{q_2}z+z^{q_3}x$ & $X_2X_3^{q_2}$ & $X_3X_1^{q_3}$ & $X_1X_2^{q_1}$ \\
\hline
\end{tabular}
\end{center}
\caption{The image of the generators in $R_{A_{G_f}}$}\label{TabGenerator}
\end{table}
\end{proof}
Next, we give another proof which is more geometric. 
\begin{lemma}\label{lem:genus}
The genus of the underlying smooth projective curve $C_{G_f}$ 
of the stack ${\mathcal C}_{G_f}$ is zero.
\end{lemma}
\begin{proof}
We shall calculate the dimension of the space of holomorphic $1$-forms 
on the curve $C_{G_f}$.
Recall that any holomorphic $1$-form on the curve $C_{G_f}$ 
is of the following form: 
\[
\omega(g):=g(x,y,z)\frac{xdy\wedge dz-ydx\wedge dz+zdx\wedge dy}{df},
\]
where $g(x,y,z)$ is a weighted homogeneous representative of an element in the 
Jacobian ring $Jac(f)$. Note also that $\omega(g)$ must be $G_f$-invariant. 
By a case by case study based on Table \ref{TabAppendix1}, 
we can show that $g(x,y,z)=0$ in $Jac(f)$.
\end{proof}
\begin{remark}
The above proof is a generalization of the one of \cite[Theorem 3]{s:1}.
\end{remark}
\begin{lemma}\label{lem:isotropy}
On the underlying smooth projective curve $C_{G_f}$ of the stack ${\mathcal C}_{G_f}$,
there exist at most three isotropic points of orders $\alpha_1,\alpha_2,\alpha_3$ 
given in Table \ref{TabAppendix2}.
\end{lemma}
\begin{proof}
Since each element $(\lambda_1,\lambda_2, \lambda_3)\in G_f$ acts on $\CC^3$ diagonally
\[
(x,y,z)\mapsto (\lambda_1 x,\lambda _2 y, \lambda_3 z),
\]
any isotropic point must be contained in the subvariety $\{xyz=0\}\subset C_{G_f}$.
By a case by case study based on Table \ref{TabAppendix1}, 
we first see that there are at most three isotropic points. 
Then, by considering the equation 
\[
\prod_{j=1}^3 \lambda_j ^{E_{1j}}=\prod_{j=1}^3 \lambda_j ^{E_{2j}}
=\prod_{j=1}^3\lambda_j^{E_{3j}}
\]
at each isotropic point, we can show that the isotropy group is a cyclic group.
The triple of orders of these isotropy groups coincides with the one
in Table \ref{TabAppendix2}.
\end{proof}
One sees that Theorem \ref{thm:orbifold} now follows from 
the above Lemma \ref{lem:genus} and Lemma \ref{lem:isotropy}.
\end{proof}
\section{Gabrielov numbers for $f(x,y,z)$}

\begin{definition}
A polynomial which, by a suitable holomorphic change of coordinates, becomes the polynomial
\[ x^{\gamma_1} + y^{\gamma_2} + z^{\gamma_3} +axyz, \quad \mbox{for some } a \neq 0, \]
is called a polynomial of type $T_{\gamma_1, \gamma_2, \gamma_3}$. 
\end{definition}

\begin{definition} For a triple $(a,b,c)$ of positive integers we define
\[ \Delta(a,b,c) := abc - bc - ac -ab. \]
\end{definition}

\begin{remark} If
\[ \Delta(\gamma_1, \gamma_2, \gamma_3) > 0\]
then a polynomial of type $T_{\gamma_1, \gamma_2, \gamma_3}$ defines a cusp singularity of this type. We do not assume this condition here.
\end{remark}

\begin{sloppypar}

\begin{remark} A polynomial of type $T_{\gamma_1, \gamma_2, \gamma_3}$ has a Coxeter-Dynkin diagram of type $T(\gamma_1, \gamma_2, \gamma_3)$ (see Fig.~\ref{FigTpqr}). This is the numbered graph encoding an intersection matrix of a distinguished basis of vanishing cycles. It corresponds to the matrix  $A=(a_{ij})$ defined by $a_{ii}=-2$, $a_{ij}=0$ if the vertices $\bullet_i$ and $\bullet_j$ are not connected, and
\[  a_{ij}=1 \Leftrightarrow \xymatrix{\bullet_i \ar@{-}[r] & \bullet_j}, \quad a_{ij}=-2 \Leftrightarrow \xymatrix{\bullet_i \ar@{==}[r] & \bullet_j} .\]
The number $(-1)^{\gamma_1+\gamma_2+\gamma_3-1}\Delta(\gamma_1, \gamma_2, \gamma_3)$ is the discriminant of the symmetric bilinear form defined by the matrix $A$, i.e.\  the determinant of $A$.
\end{remark}

\end{sloppypar}

\begin{figure}
$$
\xymatrix{ 
 & & & {\bullet} \ar@{==}[d] \ar@{-}[dr]  \ar@{-}[ldd] \ar@{}^{\gamma_1+\gamma_2+\gamma_3-1}[r] 
 & & &  \\
 {\bullet} \ar@{-}[r] \ar@{}_{\gamma_1}[d]  & {\cdots} \ar@{-}[r]  & {\bullet} \ar@{-}[r] \ar@{-}[ur]   \ar@{}_{\gamma_1+\gamma_2-2}[d] & {\bullet} \ar@{-}[dl] \ar@{-}[r] \ar@{}_{\gamma_1+\gamma_2+\gamma_3-2}[r] & {\bullet} \ar@{-}[r]  \ar@{}^{\gamma_1+\gamma_2+\gamma_3-3}[d]  & {\cdots} \ar@{-}[r]  &{\bullet} \ar@{}^{\gamma_1+\gamma_2-1}[d]   \\
& &  {\bullet} \ar@{-}[dl] \ar@{}_{\gamma_1-1}[r]  & & & &  \\
 & {\cdots} \ar@{-}[dl] & & & & & \\
{\bullet}  \ar@{}_{1}[r] & & & & & &
  }
$$
\caption{The graph $T(\gamma_1, \gamma_2, \gamma_3)$} \label{FigTpqr}
\end{figure}
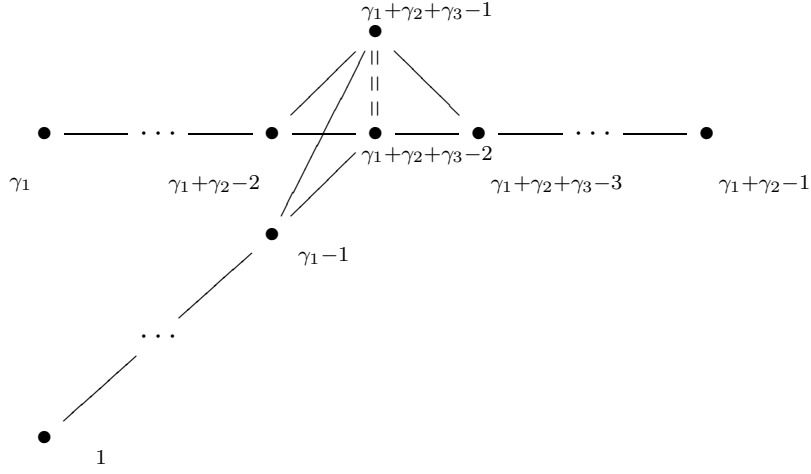

\begin{theorem} \label{thm:cusp}
Let $f(x,y,z)$ be an invertible polynomial. We associate to $f$ the numbers $\gamma_1, \gamma_2, \gamma_3$ according to Table~\ref{TabGabrielov}. \begin{itemize}
\item[{\rm (i)}] If $\Delta(\gamma_1, \gamma_2, \gamma_3) < 0$ then there exists a deformation of the polynomial $f(x,y,z) + xyz$ to a polynomial of type $T_{\gamma_1, \gamma_2, \gamma_3}$.
\item[{\rm (ii)}] If $\Delta(\gamma_1, \gamma_2, \gamma_3)=0$ then for some $a \neq 0$ the polynomial $f(x,y,z) + axyz$ is a 
polynomial of type $T_{\gamma_1, \gamma_2, \gamma_3}$.
\item[{\rm (iii)}] If $\Delta(\gamma_1, \gamma_2, \gamma_3)>0$ then the polynomial $f(x,y,z) + xyz$ is a polynomial of type $T_{\gamma_1, \gamma_2, \gamma_3}$.
\end{itemize}
\begin{table}[h]
\begin{center}
\begin{tabular}{|c||c|c|}
\hline
{\rm Type} & $f(x,y,z)$ & $\left(\gamma_1, \gamma_2, \gamma_3\right)$ \\
\hline
\hline
I & $x^{p_1}+y^{p_2}+z^{p_3}$ & $\left(p_1,p_2,p_3\right)$\\
\hline
II & $x^{p_1}+y^{p_2}+yz^{\frac{p_3}{p_2}}$ & $\left(p_1, p_2, (\frac{p_3}{p_2}-1)p_1\right)$\\
\hline
III & $x^{p_1}+zy^{q_2+1}+yz^{q_3+1}$ &  $\left(p_1,p_1q_2,p_1q_3\right)$\\
\hline
IV & $x^{p_1}+xy^{\frac{p_2}{p_1}}+yz^{\frac{p_3}{p_2}}$ 
& $\left(p_1, (\frac{p_3}{p_2}-1)p_1, \frac{p_3}{p_1}-\frac{p_3}{p_2}+1\right)$\\
\hline
V & $x^{q_1}y+y^{q_2}z+z^{q_3}x$ & $\left(q_2q_3-q_2+1,q_3q_1-q_3+1,q_1q_2-q_1+1\right)$\\
\hline
\end{tabular}
\end{center}
\caption{Gabrielov numbers for  $f$}\label{TabGabrielov}
\end{table}
\end{theorem}

\begin{definition}
The numbers $(\gamma_1, \gamma_2, \gamma_3)$ in Theorem \ref{thm:cusp} 
are called the {\em Gabrielov numbers} of $f$ and the tuple is denoted by 
$\Gamma_f$.
\end{definition}

\begin{proof}[Proof of Theorem \ref{thm:cusp}]
We prove this theorem case by case. We indicate in each case the necessary coordinate transformations.

\noindent {\bf Type I}. No transformation is needed.

\noindent {\bf Type II}. $x \mapsto x - z^{\frac{p_3}{p_2}-1}, \ y \mapsto y , \ z \mapsto z$.

\noindent {\bf Type III}. $x \mapsto x - y^{q_2}-z^{q_3}, \ y \mapsto y, \ z \mapsto z$.

\noindent {\bf Type IV}. $x \mapsto x- z^{\frac{p_3}{p_2}-1}, \ y \mapsto y, \ z \mapsto z - y^{\frac{p_2}{p_1}-1}$.

\noindent {\bf Type V}. $x \mapsto x- y^{q_2-1}, \ y \mapsto y -z^{q_3-1}, \ z \mapsto z-x^{q_1-1}$.
\end{proof}

\begin{corollary}\label{cor13}
Let $f(x,y,z)$ be an invertible polynomial with Gabrielov numbers $\Gamma_f=(\gamma_1, \gamma_2, \gamma_3)$.
\begin{itemize}
\item[{\rm (i)}] If $\Delta(\Gamma_f) < 0$ then the polynomial of type $T_{\gamma_1, \gamma_2, \gamma_3}$ deforms to $f$.
\item[{\rm (ii)}] If $\Delta(\Gamma_f) >0$ then the singularity given by $f(x,y,z)=0$ deforms to a cusp singularity of type $T_{\gamma_1, \gamma_2, \gamma_3}$.
\end{itemize}
\end{corollary}

If a singularity $f$ deforms to a singularity $g$ then a Coxeter-Dynkin diagram of $g$ can be extended to a Coxeter-Dynkin diagram of $f$.  Therefore we obtain:

\begin{corollary} 
Let $f(x,y,z)$ be an invertible polynomial with Gabrielov numbers $\Gamma_f=(\gamma_1, \gamma_2, \gamma_3)$. 
\begin{itemize}
\item[{\rm (i)}] If $\Delta(\Gamma_f) < 0$ then a Coxeter-Dynkin diagram of $f$ is contained in the graph of type $T(\gamma_1, \gamma_2, \gamma_3)$. More precisely, a Coxeter-Dynkin diagram of $f$ is of type ADE and the graph $T(\gamma_1, \gamma_2, \gamma_3)$ is the corresponding extended affine ADE-diagram.
\item[{\rm (ii)}] If $\Delta(\Gamma_f) =0$ then the graph $T(\gamma_1, \gamma_2, \gamma_3)$ is a Coxeter-Dynkin diagram of $f$.
\item[{\rm (iii)}] If $\Delta(\Gamma_f)>0$ then the graph of type $T(\gamma_1, \gamma_2, \gamma_3)$ can be extended to a Coxeter-Dynkin diagram of $f$.
\end{itemize}
\end{corollary}
Moreover, Corollary \ref{cor13} gives us a relation (a semi-orthogonal decomposition
theorem) between the Fukaya categories
$D^b{\rm Fuk}^\to(f)$ and
$D^b{\rm Fuk}^\to(T_{\gamma_1,\gamma_2,\gamma_3})$,
which is mirror dual to the one proven by Orlov \cite{o:1}.

\section{Strange Duality}

Now we are ready to state our main theorem in this paper.
\begin{theorem}\label{thm:main}
Let $f(x,y,z)$ be an invertible polynomial.
Then we have 
\begin{equation}
A_{G_f}=\Gamma_{f^t},\quad A_{G_{f^t}}=\Gamma_{f}.
\end{equation}
Namely, the Dolgachev numbers $A_{G_f}$ for the pair $(f,G_f)$ coincide with 
the Gabrielov numbers $\Gamma_{f^t}$ for the Berglund--H\"{u}bsch transpose $f^t$ of $f$, 
and the Dolgachev numbers $A_{G_{f^t}}$ for the pair $(f^t,G_{f^t})$ coincide with 
the Gabrielov numbers $\Gamma_f$ for $f$.
\begin{table}[h]
\begin{center}
\begin{tabular}{|c||c||c|}
\hline
{\rm Type} & $A_{G_f}=\left(\alpha_1,\alpha_2,\alpha_3\right)=\Gamma_{f^t}$ 
& $\Gamma_{f}=\left(\gamma_1,\gamma_2,\gamma_3\right)=A_{G_{f^t}}$\\
\hline
\hline
I & $\left(p_1,p_2,p_3\right)$ & $\left(p_1,p_2,p_3\right)$\\
\hline
II & $\left(p_1,\frac{p_3}{p_2},(p_2-1)p_1\right)$ 
& $\left(p_1, p_2, (\frac{p_3}{p_2}-1)p_1\right)$\\
\hline
III & $\left(p_1,p_1q_2,p_1q_3\right)$ & $\left(p_1,p_1q_2,p_1q_3\right)$\\
\hline
IV & $\left(\frac{p_3}{p_2}, (p_1-1)\frac{p_3}{p_2}, p_2-p_1+1\right)$ 
& $\left(p_1, (\frac{p_3}{p_2}-1)p_1, \frac{p_3}{p_1}-\frac{p_3}{p_2}+1\right)$\\
\hline
V & $\left(q_2q_3-q_3+1,q_3q_1-q_1+1,q_1q_2-q_2+1\right)$ 
& $\left(q_2q_3-q_2+1,q_3q_1-q_3+1,q_1q_2-q_1+1\right)$ \\
\hline
\end{tabular}
\end{center}
\caption{Strange duality}\label{TabAppendix3}
\end{table}
\end{theorem}
\begin{proof}
This can be easily checked by 
Tables \ref{TabAppendix1}, \ref{TabAppendix2}, and \ref{TabGabrielov}.
See Table \ref{TabAppendix3}.
\end{proof}
It is convenient in the next section to introduce the following:
\begin{definition}
Let $X$ and $Y$ be weighted homogeneous isolated hypersurface singularities 
of dimension $2$. 
If there exists an invertible polynomial $f(x,y,z)$ such that 
$f$ represents the singularity $X$ and $f^t$ represents 
the singularity $Y$, then $Y$ is called {\em $f$-dual} to $X$.
\end{definition}
Note that $Y$ is $f$-dual to $X$ if and only if $X$ is $f^t$-dual to $Y$. 
\begin{definition}
Let $X$ be a weighted homogeneous isolated hypersurface singularity
of dimension $2$. 
The singularity $X$ is called {\em $f$-selfdual} if $X$ is $f$-dual to $X$. 
\end{definition}
If the invertible polynomial $f$ is clear from the context, 
we shall often say ``dual" instead of ``$f$-dual" for simplicity. 
\section{Additional features of the duality}

\begin{theorem}
Let $f(x,y,z)$ be an invertible polynomial.
Then, we have 
\begin{equation}
\Delta(A_{G_f})=\Delta(A_{G_{f^t}}).
\end{equation}
\end{theorem}
\begin{proof}
This can be easily shown by direct calculation based on Table \ref{TabAppendix3}.
\end{proof}
\begin{remark}
The rational number 
\[
\chi_{G_f}:=2+\sum_{i=1}^3\left(\frac{1}{\alpha_i}-1\right)
=\sum_{i=1}^3\frac{1}{\alpha_i}-1
\]
is called the {\em orbifold Euler number} of the stack 
${\mathcal C}_{A_{G_f}}={\mathcal C}_{G_f}$.
Note that $\Delta(A_{G_f})=-\alpha_1\alpha_2\alpha_3\cdot \chi_{G_f}$.
\end{remark}
\begin{definition}
Let $f(x,y,z)$ be an invertible polynomial.
The {\em canonical system of weights} $W_f$ is the system of weights 
$(w_1,w_2,w_3;d)$ given in  Table \ref{weights}.
\begin{table}[h]
\begin{center}
\begin{tabular}{|c||c|}
\hline
Type & $W_f=\left(w_1,w_2,w_3;d\right)$ \\
\hline
\hline
I & $\left(p_2p_3,p_3p_1,p_1p_2;p_1p_2p_3\right)$\\
\hline
II & $\left(p_3,\frac{p_1p_3}{p_2},(p_2-1)p_1;p_1p_3\right)$ \\
\hline
III & $\left(p_2,p_1q_3,p_1q_2;p_1p_2\right)$\\
& $(p_2+1=(q_2+1)(q_3+1))$\\
\hline
IV & $\left(\frac{p_3}{p_1},(p_1-1)\frac{p_3}{p_2},p_2-p_1+1;p_3\right)$\\
\hline
V & $\left(q_2q_3-q_3+1,q_3q_1-q_1+1,q_1q_2-q_2+1;q_1q_2q_3+1\right)$\\
\hline
\end{tabular}
\end{center}
\caption{Canonical system of weights attached to $f$}\label{weights}
\end{table}
\end{definition}
Note that a canonical system of weights is non-reduced in general.

\begin{definition}
Let $f(x,y,z)$ be an invertible polynomial and $W_f=(w_1,w_2,w_3;d)$ the canonical system of weights attached to $f$. Define
\[ c_f:= {\rm gcd}(w_1, w_2 ,w_3,d). \]
\end{definition}

\begin{definition}
Let $f(x,y,z)$ be a weighted homogeneous polynomial and 
$W:=(w_1,w_2,w_3;d)$ be a system of weights attached to $f$. 
The integer
\[
a_{W} := d- w_1-w_2-w_3
\]
is called the {\em Gorenstein parameter} of $W$. 
\end{definition}
\begin{remark}
The Gorenstein parameter $a_W$ is denoted by $-\epsilon_W$ in \cite{s:2}.
\end{remark}
\begin{remark}
Note that the integer $\Delta(A_{G_f})$ can also be regarded as 
the Gorenstein parameter of the $\ZZ$-graded ring 
$R_{A_{G_f}}=\CC[X_1,X_2,X_3]/(X_1^{\alpha_1}+X_2^{\alpha_2}+X_3^{\alpha_3})$ 
with respect to the system of weights 
$(\alpha_2\alpha_3,\alpha_3\alpha_1,\alpha_1\alpha_2;\alpha_1\alpha_2\alpha_3)$ 
attached to the polynomial 
$X_1^{\alpha_1}+X_2^{\alpha_2}+X_3^{\alpha_3}$.
\end{remark}

\begin{theorem}
Let $f(x,y,z)$ be an invertible polynomial.
Let $W_f$ and $W_{f^t}$ be the canonical systems of weights attached to $f$ and $f^t$.
Then, we have 
\begin{equation}
a_{W_f}=a_{W_{f^t}}.
\end{equation}
\end{theorem}
\begin{proof}
One can easily show this by direct calculation based on Table \ref{weights} 
and Table \ref{dualweights}.
\begin{table}[h]
\begin{center}
\begin{tabular}{|c||c|}
\hline
Type & $W_{f^t}=\left(w_1,w_2,w_3;d\right)$ \\
\hline
\hline
I & $\left(p_2p_3,p_3p_1,p_1p_2;p_1p_2p_3\right)$\\
\hline
II & $\left(p_3,(\frac{p_3}{p_2}-1)p_1,p_1p_2;p_1p_3\right)$ \\
\hline
III & $\left(p_2,p_1q_3,p_1q_2;p_1p_2\right)$\\
& $(p_2+1=(q_2+1)(q_3+1))$\\
\hline
IV & $\left(\frac{p_3}{p_1}-\frac{p_3}{p_2}+1,(\frac{p_3}{p_2}-1)p_1,p_2;p_3\right)$\\
\hline
V & $\left(q_2q_3-q_2+1,q_3q_1-q_3+1,q_1q_2-q_1+1;q_1q_2q_3+1\right)$\\
\hline
\end{tabular}
\end{center}
\caption{Canonical system of weights attached to ${f^t}$}\label{dualweights}
\end{table}
\end{proof}
We have the following combinatorial formula to calculate 
the canonical systems of weights attached to ${f^t}$:
\begin{theorem}
Let $f(x_1,x_2 ,x_3)=\sum_{i=1}^n\prod_{j=1}^3x_j^{E_{ij}}$ be an invertible polynomial.
Let $E_i$, $i=1,2,3$ be the $3$-dimensional vectors defined by $(E_i)_j:=E_{ij}$. 
Then the canonical system of weights $W_{f^t}$ attached to $f^t$ is given by the formula
\[
W_{f^t}=\left(
(E_2\times E_3)\cdot 
\begin{pmatrix}
1 \\ 1\\ 1
\end{pmatrix},
(E_3\times E_1)\cdot 
\begin{pmatrix}
1 \\ 1\\ 1
\end{pmatrix},
(E_1\times E_2)\cdot 
\begin{pmatrix}
1 \\ 1\\ 1
\end{pmatrix};\det (E_{ij})
\right),
\]
where we denote by $\times$ and $\cdot $ the exterior product and the inner product of 
$3$-dimensional vectors respectively. 
\end{theorem}
\begin{proof}
One can easily show this by direct calculation.
\end{proof}
For an invertible polynomial $f(x,y,z)$, 
the ring $R_f:=\CC[x,y,z]/(f)$ is a $\ZZ$-graded ring with respect to the canonical system of 
weights $(w_1,w_2,w_3;d)$ attached to $f$. 
Therefore, we can consider the decomposition of $R_f$ 
as a $\ZZ$-graded $\CC$-vector space:
\[
R_f:=\bigoplus_{k\in\ZZ_{\ge 0}} R_{f,k}, \quad R_{f,k}:=\left\{g\in R_f~\left|~
w_1x\frac{\partial g}{\partial x}+w_2y\frac{\partial g}{\partial y}
+w_3x\frac{\partial g}{\partial z}=k g\right.\right\}.
\]
\begin{definition}
Let $f(x,y,z)$ be an invertible polynomial. 
The formal power series
\begin{equation}
p_f(t):=\sum_{k\ge 0} (\dim_\CC R_{f,k}) t^k
\end{equation}
is called the {\em Poincar\'e series} of the $\ZZ$-graded coordinate ring $R_f$
with respect to the canonical system of weights $(w_1,w_2,w_3;d)$ attached to $f$.
\end{definition}
It is easy to see that for an invertible polynomial $f(x,y,z)$ with 
canonical system of weights $(w_1,w_2,w_3;d)$
the Poincar\'e series $p_f(t)$ is given by  
\[
p_f(t)=\frac{(1-t^{w_1})(1-t^{w_2})(1-t^{w_3})}{(1-t^d)}
\]
and defines a rational function.
In order to simplify some notations, we shall denote the rational function of the form 
\[
\frac{\prod_{l=1}^L(1-t^{i_k})}{\prod_{m=1}^M(1-t^{j_l})}
\]
by
\[
i_1\cdot i_2\cdot\ \dots\ \cdot i_L\left/j_1\cdot j_2\cdot\ \dots\ \cdot j_M\right. .
\]
For example, the Poincar\'e series $p_f(t)$ is denoted by 
$w_1\cdot w_2\cdot w_3\left/d \right.$.
\begin{definition}
Let $f(x,y,z)$ be an invertible polynomial. 
Let $p_f(t)$ be the Poincar\'e series of the $\ZZ$-graded coordinate ring 
$R_f$ with respect to the canonical system of weights attached to $f$ and 
$A_{G_f}=(\alpha_1,\alpha_2,\alpha_3)$ be the Dolgachev numbers of the pair $(f,G_f)$.
The rational function
\[
\phi_{G_f}(t) := p_f(t)(1-t)^{2}\prod_{i=1}^3
\frac{1-t^{\alpha_i}}{1-t}
\]
is called the {\em characteristic function} of $f$. 
The functions $\phi_{G_f}(t)$ are listed in Table \ref{phi}.
\begin{table}[h]
\begin{center}
\begin{tabular}{|c||c|}
\hline
Type & $\phi_{G_f}(t)$\\
\hline
\hline
I & $p_1\cdot p_2\cdot p_3\cdot p_1p_2p_3\left/1\cdot p_2p_3\cdot p_3p_1\cdot p_1p_2 \right. $ \\
\hline
II & $p_1\cdot \frac{p_3}{p_2}\cdot p_1p_3\left/1\cdot p_3\cdot \frac{p_1p_3}{p_2}\right. $  \\
\hline
III & $p_1\cdot p_1p_2\left/1\cdot p_2\right. $ \\
\hline
IV & $\frac{p_3}{p_2}\cdot p_3\left/1\cdot \frac{p_3}{p_1}\right. $ \\
\hline
V & $q_1q_2q_3+1\left/1\right. $  \\
\hline
\end{tabular}
\end{center}
\caption{Characteristic function $\phi_{G_f}(t)$}\label{phi}
\end{table}
\end{definition}
Here, we recall the notion of Saito's $*$-duality (see \cite{s:2}):
\begin{definition}
Let $d$ be a positive integer and $\phi(t)$ be a rational function of the form  
\[
\phi(t)=\prod_{i|d}(1-t^{i})^{e(i)}, \quad e(i)\in\ZZ.
\]
The {\em Saito dual} $\phi^* (t)$ of $\phi(t)$ is the rational function 
given by 
\[
\phi^*(t)=\prod_{i|d}(1-t^{\frac{d}{i}})^{-e(i)}.
\]
One easily sees that $(\phi^*)^*(t)=\phi(t)$. 
\end{definition}
The characteristic function $\phi_{G_f}(t)$ may not be a polynomial in general, 
however, by Theorem 1 in \cite{E02}, we have the following:
\begin{theorem}\label{Monodromy}
Let $f(x,y,z)$ be an invertible polynomial whose canonical system of weights
$(w_1,w_2,w_3;d)$ is reduced.
Then the characteristic function $\phi_{G_f}(t)$ is a polynomial.
Moreover, its Saito dual $\phi_{G_f}^*(t)$ is the characteristic polynomial of 
the Milnor monodromy of the singularity $f$.
\qed
\end{theorem}
Here we use a normalized version of the characteristic polynomial.
If $\tau$ denotes the Milnor monodromy of the singularity $f$, 
then its characteristic polynomial is 
\[
\Phi_f(t):=\det (1-\tau^{-1}t).
\]
Even if $f(x,y,z)$ is an invertible polynomial whose canonical system of weights
$(w_1,w_2,w_3;d)$ is reduced, 
the canonical system of weights of its transpose $f^t$ may not be reduced.
We have the following property of our characteristic functions:
\begin{theorem}\label{Saitoduality}
Let $f(x,y,z)$ be an invertible polynomial. Then we have the Saito duality
\[
\phi_{G_f}^*(t)=\phi_{G_{f^t}}(t).
\]
In particular, if the canonical system of weights attached to $f$ is reduced,
then the polynomial $\phi_{G_f}(t)$ is the characteristic polynomial of an operator $\tau$ such that $\tau^{c_{f^t}}$ is the monodromy of the singularity $f^t$.
\end{theorem}

\begin{proof}
One can easily check the first statement by direct calculations. 

The characteristic polynomials of the monodromy of $f^t$ can be computed using Varchenko's method \cite{V}. They are listed in Table~\ref{phimon}.
In order to prove the second statement, assume that the canonical system of weights attached to $f$ is reduced. If the canonical system of weights attached to $f^t$ is also reduced then the second statement follows from Theorem~\ref{Monodromy}. Otherwise,
let $c=c_{f^t}$.  The case that $c > 1$ can only occur for type II (with $c_1=c$ and $c_2=1$), IV, or V.  Let $\zeta:= e^{\frac{2 \pi i}{c}}$ be a primitive $c$-th root of unity. There is the following relation between the characteristic polynomial of an operator $\tau$ and of the operator $\tau^c$:
\[ \det(1-\tau^{-c}t^c) = \prod_{i=0}^{c-1} \det(1-\tau^{-1}\zeta^i t). \]
Using this relation and Table~\ref{phimon}, one can easily show the second statement for the remaining cases. 
\begin{table}[h]
\begin{center}
\begin{tabular}{|c||c|}
\hline
Type & $\Phi_{f^t}(t)$ \\
\hline
\hline
I & $p_1\cdot p_2\cdot p_3\cdot \left(\frac{p_1p_2p_3}{c}\right)^c\left/1\cdot \left(\frac{p_2p_3}{c_1}\right)^{c_1} \cdot \left(\frac{p_3p_1}{c_2} \right)^{c_2} \cdot \left( \frac{p_1p_2}{c_3} \right)^{c_3}\right. $\\
 & $c_1 = {\rm gcd}(p_2, p_3), c_2 = {\rm gcd}(p_1, p_3), c_3 = {\rm gcd}(p_1, p_2)$ \\
\hline
II & $p_1\cdot \frac{p_3}{p_2}\cdot \left( \frac{p_1p_3}{c} \right)^c \left/1\cdot \left( \frac{p_3}{c_1} \right)^{c_1} \cdot \left( \frac{p_1p_3}{p_2c_2} \right)^{c_2} \right. $ \\
 & $c_1 = {\rm gcd}( p_2, \frac{p_3}{p_2}-1), c_2 = {\rm gcd}(p_1, \frac{p_3}{p_2})$ \\
\hline
III & $p_1\cdot \left(\frac{p_1p_2}{c} \right)^c \left/1\cdot \left(\frac{p_2}{c_1} \right)^{c_1} \right. $\\
 & $c_1 = {\rm gcd}( q_2,q_3)$ \\
\hline
IV & $\frac{p_3}{p_2}\cdot \left(\frac{p_3}{c} \right)^c \left/1\cdot \left(\frac{p_3}{p_1c_1} \right)^{c_1}\right. $\\
 & $c_1 = {\rm gcd}(\frac{p_2}{p_1}, \frac{p_3}{p_2}-1)$ \\
\hline
V & $\left(\frac{q_1q_2q_3+1}{c} \right)^c \left/1\right. $\\
\hline
\end{tabular}
\end{center}
\caption{Characteristic polynomial of the monodromy of $f^t$}\label{phimon}
\end{table}

\end{proof}
\begin{remark}
Theorem \ref{Saitoduality} is already shown in \cite{t:3} for the 
special case when both the canonical systems of weights for $f$ and $f^t$ are reduced.
\end{remark}
\section{Examples, Coxeter--Dynkin diagrams}

We now show how the weighted homogeneous singularities of Arnold's classification of singularities fit into our scheme.

\begin{definition} Let $f(x,y,x)$ be an invertible polynomial whose canonical system 
of weights is reduced, $\alpha_1, \alpha_2, \alpha_3$ be the Dolgachev numbers of $f$, and $a_{W_f}$ be the Gorenstein parameter of $W_f$. 
We define positive integers $\beta_i$, $0 < \beta_i < \alpha_i$, by 
\[ \beta_i a_{W_f} \equiv 1 \, {\rm mod}\, \alpha_i, \quad i=1,2,3. \]
The numbers $(\alpha_1, \beta_1), (\alpha_2, \beta_2), (\alpha_3, \beta_3)$ are called the {\em orbit invariants} of $f$.
\end{definition}

\begin{remark}
Since the weight system is assumed to be reduced, $G_f \cong \CC^\ast$ and we have the usual $\CC^\ast$-action. The numbers $(\alpha_1, \beta_1), (\alpha_2, \beta_2), (\alpha_3, \beta_3)$ are just the usual orbit invariants of the $\CC^\ast$-action (see \cite{Dolgachev83}). 
\end{remark}

We now consider the classification of the singularities 
defined by invertible polynomials according to the Gorenstein parameter $a_{W_f}$.

The invertible polynomials $f(x,y,z)$ with $a_{W_f} < 0$ define the simple singularities. These invertible polynomials together with the corresponding Dolgachev and Gabrielov numbers are given in Table~\ref{TabADE}. They are all self-dual.

\begin{table}[h]
\begin{center}
\begin{tabular}{|c|c|c||c|c|c|}  \hline
Name & $\alpha_1,\alpha_2, \alpha_3$  & $f$ & $f^t$ & $\gamma_1, \gamma_2, \gamma_3$ & Dual
\\ \hline
$A_{kl}$ & $1,k,kl-k+1$ & $xy+y^kz+z^lx$ & $zx+xy^k+yz^l$ & $1,k,kl-k+1$ & $A_{kl}$ \\
$D_{2k}$ & $2,2,2k-2$ & $x^2+zy^2+yz^k$ & $x^2+zy^2+yz^k$ & $2,2,2k-2$ & $D_{2k}$ \\
$D_{2k+1}$ & $2,2,2k-1$ & $x^2+xy^k+yz^2$ & $x^2y+y^kz+z^2$ & $2,2,2k-1$ & $D_{2k+1}$ \\
$E_6$ & $2,3,3$ & $x^3+y^2+yz^2$ & $x^3+y^2z+z^2$ & $2,3,3$ & $E_6$ \\
$E_7$ & $2,3,4$ & $x^2+y^3+yz^3$ & $x^2+y^3z+z^3$ & $2,3,4$ & $E_7$ \\
$E_8$ & $2,3,5$ & $x^2+y^3+z^5$ & $x^2+y^3+z^5$ & $2,3,5$ & $E_8$ \\
\hline
\end{tabular}
\end{center}
\caption{Simple singularities}\label{TabADE}
\end{table}

The invertible polynomials $f(x,y,z)$ with $a_{W_f} =0$ define the simply elliptic singularities. They are exhibited in
Table~\ref{TabEll}. The corresponding weight systems are not reduced. Again all polynomials are self-dual.

\begin{table}[h]
\begin{center}
\begin{tabular}{|c|c|c||c|c|c|}  \hline
Name & $\alpha_1,\alpha_2, \alpha_3$  & $f$ & $f^t$ & $\gamma_1, \gamma_2, \gamma_3$ & Dual
\\ \hline
$\widetilde{E}_6$ & $3,3,3$ & $x^3+y^3+z^3$ & $x^3+y^3+z^3$ & $3,3,3$ & $\widetilde{E}_6$ \\
$\widetilde{E}_7$ & $2,4,4$ & $x^2+y^4+z^4$ & $x^2+y^4+z^4$ & $2,4,4$ & $\widetilde{E}_7$ \\
$\widetilde{E}_8$ & $2,3,6$ & $x^2+y^3+z^6$ & $x^2+y^3+z^6$ & $2,3,6$ & $\widetilde{E}_8$ \\
\hline
\end{tabular}
\end{center}
\caption{Simply elliptic singularities}\label{TabEll}
\end{table}

Now we consider invertible polynomials $f(x,y,z)$ with $a_{W_f} >0$. 
In Table~\ref{TabArnold} we indicate the invertible polynomials for the exceptional unimodal singularities. We obtain Arnold's strange duality. Here the weight systems are all reduced and we have $a_{W_f}=1$. We obtain a Coxeter-Dynkin diagram for $f$ by adding one new vertex to the graph $T(\gamma_1,\gamma_2,\gamma_3)$ (see Figure \ref{FigTpqr}) and connecting it to the upper central vertex 
(with index $\gamma_1+\gamma_2+\gamma_3-1$) by a solid edge. Therefore our Gabrielov numbers coincide with the numbers defined by Gabrielov in this case.

\begin{table}[h]
\begin{center}
\begin{tabular}{|c|c|c||c|c|c|}  \hline
Name & $\alpha_1,\alpha_2, \alpha_3$  & $f$ & $f^t$ & $\gamma_1, \gamma_2, \gamma_3$ & Dual
\\ \hline
$E_{12}$ & $2,3,7$ & $x^2+y^3+z^7$ & $x^2+y^3+z^7$ & $2,3,7$ & $E_{12}$ \\ 
$E_{13}$ & $2,4,5$ & $x^2+y^3+yz^5$ & $x^2+zy^3+z^5$ & $2,3,8$ & $Z_{11}$ \\ 
$E_{14}$ & $3,3,4$  & $x^3+y^2+yz^4$ &  $x^3+zy^2+z^4$ & $2,3,9$ & $Q_{10}$ \\ 
$Z_{12}$ & $2,4,6$ & $x^2+zy^3+yz^4$ & $x^2+zy^3+yz^4$ & $2,4,6$ & $Z_{12}$ \\
$Z_{13}$ & $3,3,5$ & $x^2+xy^3+yz^3$ & $x^2y+y^3z+z^3$ & $2,4,7$ & $Q_{11}$ \\ 
$Q_{12}$ & $3,3,6$ & $x^3+zy^2+yz^3$ & $x^3+zy^2+yz^3$ & $3,3,6$ &  $Q_{12}$ \\ 
$W_{12}$ & $2,5,5$ & $x^5 + y^2+yz^2$ & $x^5+y^2z+z^2$ & $2,5,5$ & $W_{12}$ \\ 
$W_{13}$ & $3,4,4$ & $x^2+xy^2+yz^4$ & $x^2y+y^2z+z^4$ & $2,5,6$ & $S_{11}$ \\
$S_{12}$ & $3,4,5$ & $x^3y+y^2z+z^2x$ & $zx^3+xy^2+yz^2$ & $3,4,5$ &  $S_{12}$ \\ 
$U_{12}$ & $4,4,4$ &  $x^4+zy^2+yz^2$ & $x^4+zy^2+yz^2$ & $4,4,4$ & $U_{12}$ \\ 
\hline
\end{tabular}
\end{center}
\caption{Arnold's strange duality}\label{TabArnold}
\end{table}

Now we consider the exceptional bimodal singularities (Table~\ref{TabExBi}). They can all be given by invertible polynomials with a reduced weight system. We see that we also obtain a strange duality involving the exceptional bimodal singularities and some other singularities which are given by invertible polynomials with in general non-reduced weight systems (see Table~\ref{TabDynkin}).

\begin{table}[h]
\begin{tabular}{|c|c|c||c|c|c|}
\hline
Name & $\alpha_1,\alpha_2, \alpha_3$  & $f$   & $f^t$ & $\gamma_1, \gamma_2, \gamma_3$ & Dual \\
\hline
$E_{18}$ & $3, 3, 5$ & $x^3 +y^2+yz^5$  & $x^3 +y^2z+z^5$ & $2,3,12$ & $Q_{12}$ \\
$E_{19}$ & $2, 4, 7$ & $x^2+y^3+yz^7$  & $x^2+y^3z+z^7$  & $2,3,12$ & $Z_{1,0}$  \\
$E_{20}$ & $2, 3, 11$ & $x^2 +y^3 +z^{11}$ & $x^2 +y^3 +z^{11}$ & $2, 3, 11$ & $E_{20}$ \\
\hline
$Z_{17}$ & $3, 3, 7$ & $x^2+xy^4 + yz^3$ & $x^2y+y^4z + z^3$ & $2,4,10$ & $Q_{2,0}$  \\
$Z_{18}$ & $2, 4, 10$ & $x^2 + zy^3 +yz^6$ & $x^2 + zy^3 +yz^6$ & $2, 4, 10$ & $Z_{18}$  \\
$Z_{19}$ & $2, 3, 16$ & $x^2+y^9+yz^3$  & $x^2+y^9z+z^3$ & $2, 4, 9$ & $E_{25}$\\
\hline
$Q_{16}$ & $3, 3, 9$ & $x^3 +  zy^2+yz^4$ & $x^3 +  zy^2+yz^4$ & $3, 3, 9$ & $Q_{16}$ \\
$Q_{17}$ & $2, 4, 13$ & $x^3+ xy^5+ yz^2$ & $x^3y+ y^5z+ z^2$  & $3,3,9$  & $Z_{2,0}$  \\
$Q_{18}$ & $2, 3, 21$ & $x^3+y^8+yz^2$ & $x^3+y^8z+z^2$ & $3, 3, 8$ & $E_{30}$  \\
\hline
$W_{17}$ & $3, 5, 5$ & $x^2+xy^2+yz^5$ & $x^2y+y^2z+z^5$ & $2,6,8$ & $S_{1,0}$  \\
$W_{18}$ & $2, 7, 7$ & $x^7+ y^2+ yz^2$ & $x^7+ y^2z+ z^2$ & $2, 7, 7$ & $W_{18}$  \\
\hline
$S_{16}$ & $3, 5, 7$ & $x^4y+y^2z+z^2x$ & $zx^4+xy^2+yz^2$ & $3, 5, 7$ & $S_{16}$  \\
$S_{17}$ & $2, 7, 10$ & $x^6+xy^2+yz^2$ & $x^6y+y^2z+z^2$ & $3,6,6$ & $X_{2,0}$  \\
\hline
$U_{16}$ & $5, 5, 5$ & $x^5+zy^2+yz^2$ & $x^5+zy^2+yz^2$ & $5, 5, 5$ & $U_{16}$  \\
\hline
\end{tabular}
\caption{Strange duality of the exceptional bimodal singularities} \label{TabExBi}
\end{table}

We list the invariants $(\alpha_1, \beta_1), (\alpha_2, \beta_2), (\alpha_3, \beta_3)$ in Table~\ref{TabDynkin}.

We now indicate Coxeter-Dynkin diagrams for the bimodal exceptional singularities. Coxeter-Dynkin diagrams for these singularities were obtained in \cite{E2}. Let $f$ be an invertible polynomial defining an exceptional bimodal singularity with orbits invariants $(\alpha_1, \beta_1), (\alpha_2, \beta_2), (\alpha_3, \beta_3)$ and Gabrielov numbers $\gamma_1,\gamma_2,\gamma_3$. We define numbers $\delta_1, \delta_2, \delta_3$ by Table~\ref{TabDynkin}. One can show that there exists a Coxeter-Dynkin diagram which is obtained by an extension of a $T(\gamma_1,\gamma_2,\gamma_3)$-diagram by $a_{W_f}$ vertices in  the following way:
\begin{itemize}
\item If $a_{W_f}=2$ then  the diagram $T(\gamma_1,\gamma_2,\gamma_3)$ is extended by $\bullet_1 \mbox{---} \bullet_2$ where $\bullet_1$ is connected to the upper central vertex and $\bullet_2$ to the $\gamma_i - \delta_i-1$-th vertex from the outside of the $i$-th arm, unless $\delta_i=\gamma_i - 1$ ($i=1,2,3$). 
\item If $a_{W_f}=3$ then  the diagram $T(\gamma_1,\gamma_2,\gamma_3)$ is extended by $\bullet_1 \mbox{---} \bullet_2 \mbox{---} \bullet_3$ where $\bullet_1$ is connected to the upper central vertex and $\bullet_3$ to the $\gamma_i - \delta_i-1$-th vertex from the outside of the $i$-th arm, unless $\delta_i=\gamma_i - 1$ ($i=1,2,3$). 
\item If $a_{W_f}=5$ then  the diagram $T(\gamma_1,\gamma_2,\gamma_3)$
 is extended by $\bullet_1 \mbox{---} \bullet_2 \mbox{---} \bullet_3\mbox{---} \bullet_4\mbox{---} \bullet_5$ where $\bullet_1$ is connected to the upper central vertex and $\bullet_3$ to the $\gamma_i - \delta_i-1$-th vertex from the outside of the $i$-th arm, unless $\delta_i=\gamma_i - 1$ ($i=1,2,3$). 
\end{itemize}

\begin{table}[h]
\begin{tabular}{|c|c|c|c|c|c|c|}
\hline
Name &  $(\alpha_i, \beta_i), i=1,2,3$ & $a_{W_f}$ & $(\gamma_i, \delta_i), i=1,2,3$   & $c_{f^t}$ &  $\mu_{f^t}$  & Dual \\
\hline
$E_{18}$  & $(3,2) , (3,2),  (5,3)$ & 2  & $(2,1), (3,2), (12,8)$ &  2 & 12 & $Q_{12}$\\
$E_{19}$ & $(2,1), (4,3), (7,5)$ & 3  & $(2,1), (3,2), (12,9)$ & 3 & 15 & $Z_{1,0}$\\
$E_{20}$  & $(2,1), (3,2), (11,9)$ & 5  & $(2,1), (3,2), (11,9)$ & 1 & 20 & $E_{20}$ \\
$Z_{17}$ & $(3,2), (3,2), (7,4)$ & 2  & $(2,1), (4,3), (10,6)$ & 2 & 14 & $Q_{2,0}$ \\
$Z_{18}$  & $(2,1), (4,3), (10,7)$ & 3  &  $(2,1), (4,3), (10,7)$ & 1 & 18 & $Z_{18}$  \\
$Z_{19}$ & $(2,1), (3,2), (16,13)$ & 5  &  $(2,1), (4,3), (9,7)$ & 1 & 25 &  $E_{25}$\\
$Q_{16}$ & $(3,2), (3,2), (9,5)$ & 2  & $(3,2), (3,2), (9,5)$ & 1 &  16 & $Q_{16}$\\
$Q_{17}$ & $(2,1), (4,3), (13,9)$ & 3  &  $(3,2), (3,2), (9,6)$  & 3 & 21 &  $Z_{2,0}$ \\
$Q_{18}$ & $(2,1), (3,2), (21,17)$ & 5  & $(3,2), (3,2), (8,6)$ & 1 & 30 &  $E_{30}$ \\
$W_{17}$ & $(3,2), (5,3), (5,3)$ & 2 & $(2,1), (6,4), (8,5)$ & 2 & 14 & $S_{1,0}$\\
$W_{18}$ & $(2,1), (7,5), (7,5)$ & 3  &  $(2,1), (7,5), (7,5)$  & 1 &  18 & $W_{18}$\\
$S_{16}$ & $(3,2), (5,3), (7,4)$ & 2  & $(3,2), (5,3), (7,4)$ & 1 & 16 &  $S_{16}$ \\
$S_{17}$ & $(2,1), (7,5), (10,7)$ & 3  & $(3,2),(6,4), (6,4)$ & 3 & 21 &  $X_{2,0}$\\
$U_{16}$  & $(5,3), (5,3), (5,3)$ & 2  &  $(5,3), (5,3), (5,3)$ & 1 & 16 & $U_{16}$ \\
\hline
\end{tabular}
\caption{Invariants of the singularities involved} \label{TabDynkin}
\end{table}

We now consider the characteristic functions $\phi_{G_f}(t)$ and $\phi_{G_{f^t}}(t)$ in this case. They are listed in Table~\ref{TabCharBiEx}. By Theorem~\ref{Monodromy} and Theorem~\ref{Saitoduality} the characteristic function $\phi_{G_{f^t}}(t)$  is the characteristic polynomial of 
the Milnor monodromy of the singularity $f$. Moreover, the characteristic function $\phi_{G_f}$ is the characteristic polynomial of an operator $\tau$ such that $\tau^{c_{f^t}}$ is the monodromy of the singularity $f^t$. 
\begin{table}[h]
\begin{center}
\begin{tabular}{|c|c||c|c|}
\hline
Name &  $\phi_{G_f}$ & $\phi_{G_{f^t}}$  & Dual \\
\hline
$E_{18}$  & $3\cdot 5 \cdot 30/1 \cdot 10 \cdot 15$  & $2\cdot 3 \cdot 30/1 \cdot 10 \cdot 6$  & $Q_{12}$ \\
$E_{19}$  & $2 \cdot 7 \cdot 42/1 \cdot 14\cdot 21$ & $2 \cdot 3 \cdot 42/1 \cdot 6 \cdot 21$  & $Z_{1,0}$  \\
$E_{20}$  & $2 \cdot 3 \cdot 11 \cdot 66/1 \cdot 6 \cdot 22 \cdot 33$ & $2 \cdot 3 \cdot 11 \cdot 66/1 \cdot 6 \cdot 22 \cdot 33$  & $E_{20}$ \\
\hline
$Z_{17}$  & $3 \cdot 24/1 \cdot 12$ & $2 \cdot 24/1 \cdot 8$ & $Q_{2,0}$  \\
$Z_{18}$  & $2 \cdot 34/1 \cdot 17$  & $2 \cdot 34/1 \cdot 17$  & $Z_{18}$  \\
$Z_{19}$  & $2 \cdot 3 \cdot 54/1 \cdot 6 \cdot 27$ & $2 \cdot 3 \cdot  54/1 \cdot 6 \cdot 27$ & $E_{25}$\\
\hline
$Q_{16}$  & $3 \cdot 21/1 \cdot 7$ & $3 \cdot 21/1 \cdot 7$  & $Q_{16}$ \\
$Q_{17}$ & $2 \cdot 30/1 \cdot 10$ & $3 \cdot 30/1 \cdot 15$ & $Z_{2,0}$  \\
$Q_{18}$ & $2 \cdot 3 \cdot 48/1 \cdot 6 \cdot 16$ & $2 \cdot 3 \cdot 48/1 \cdot 6 \cdot 16$ & $E_{30}$\\
\hline
$W_{17}$ & $5 \cdot 20/1 \cdot 10$ & $2 \cdot 20/1 \cdot 4$ & $S_{1,0}$  \\
$W_{18}$  & $2 \cdot 7 \cdot 28/1 \cdot 4 \cdot 14$ & $2 \cdot 7 \cdot 28/1 \cdot 4 \cdot 14$ & $W_{18}$\\
\hline
$S_{16}$ & $17/1$ & $17/1$  & $S_{16}$  \\
$S_{17}$  & $2 \cdot 24/1 \cdot 4$ & $6 \cdot 24/1 \cdot 12$  & $X_{2,0}$  \\
\hline
$U_{16}$  & $5 \cdot 15/1 \cdot 3$ & $5 \cdot 15/1 \cdot 3$  & $U_{16}$  \\
\hline
\end{tabular}
\end{center}
\caption{Polynomials $\phi_{G_f}$ and $\phi_{G_{f^t}}$} \label{TabCharBiEx}
\end{table}

In Table~\ref{TabExBi}, pairs $(f,f^t)$ with $c_f=1$, $c_{f^t}>1$ only occur for Types II and IV. 
As already noted in the proof of Theorem~\ref{Saitoduality}, such pairs can only exist for Types II, IV, or V.  
Here is an example of such a pair for Type V.

\begin{example} Let $f(x,y,z)=x^2y+y^3z+z^4x$ with canonical system of weights $W_f=(9,7,4;25)$. Then $f^t(x,y,z)=zx^2+xy^3+yz^4$, $W_{f^t}=(10,5,5;25)$. Here $c_f=1$ but $c_{f^t}=5$.
\end{example}

In Table~\ref{TabExBi}, three of the 6 weighted homogeneous singularities of the bimodal series appear. 
For completeness we also indicate invertible polynomials for the remaining three in Table~\ref{TabBi}. Here both canonical systems of weights are non-reduced. In the case $J_{3,0}$, $\phi_{G_f}(t)$ is a polynomial, but not the characteristic polynomial of an operator $\tau$ such that $\tau^{c_{f^t}}$ is the monodromy of $Z_{13}$. In this case, $\phi_{G_{f^t}}(t)$ is not a polynomial. In the remaining cases, $\phi_{G_f}(t)=\phi^\ast_{G_f}(t)$ is not a polynomial.

\begin{table}[h]
\begin{center}
\begin{tabular}{|c|c|c||c|c|c|}  \hline
Name & $\alpha_1,\alpha_2, \alpha_3$  & $f$ & $f^t$ & $\gamma_1, \gamma_2, \gamma_3$ & Dual
\\ \hline
$J_{3,0}$ & $2,4,6$ & $x^2+y^3+yz^6$ & $x^2+y^3z+z^6$ & $2,3,10$ & $Z_{13}$ \\
$W_{1,0}$ & $2,6,6$ & $x^6+y^2+yz^2$ & $x^6+y^2z+z^2$ & $2,6,6$ & $W_{1,0}$ \\
$U_{1,0}$ & $3,3,6$ & $x^3+y^3+yz^3$ & $x^3+y^3z+z^3$ & $3,3,6$ & $U_{1,0}$ \\
\hline
\end{tabular}
\end{center}
\caption{Bimodal  series}\label{TabBi}
\end{table}


\end{document}